\documentclass{elsarticle}

\usepackage{amsmath,amssymb,amsthm,amscd}
\usepackage{bm}
\usepackage{a4wide}

\journal{Math. Inequal. Appl. (accepted 28/Sept/2023).}

\newtheorem{theorem}{Theorem}[section]
\newtheorem{lemma}[theorem]{Lemma}

\theoremstyle{definition}
\newtheorem{definition}[theorem]{Definition}
\newtheorem{example}[theorem]{Example}
\newtheorem{remark}[theorem]{Remark}

\numberwithin{equation}{section}
\allowdisplaybreaks[4]

\begin{document}

\begin{frontmatter}

\title{Discrete Opial type inequalities for interval-valued functions}

\author[rvt1]{Dafang Zhao}
\ead{dafangzhao@163.com}

\author[rvt1]{Xuexiao You}
\ead{youxuexiao@126.com}

\author[focal]{Delfim F. M. Torres\corref{cor1}}
\ead{delfim@ua.pt}

\cortext[cor1]{Corresponding author.}

\address[rvt1]{School of Mathematics and Statistics,
Hubei Normal University, Huangshi 435002, P. R. China.}

\address[focal]{Center for Research and Development in Mathematics and Applications (CIDMA),\\
Department of Mathematics, University of Aveiro, 3810-193 Aveiro, Portugal.\\[0.3cm]
(Accepted 28/Sept/2023 to Math. Inequal. Appl.)}

\date{}

% -----------------------------------------------------------------

\begin{abstract}
We introduce the forward (backward) gH-difference operator of interval sequences,
and establish some new discrete Opial type inequalities for interval-valued functions.
Further, we obtain generalizations of classical discrete Opial type inequalities.
Some examples are presented to illustrate our results.
\end{abstract}


\end{frontmatter}

% -----------------------------------------------------------------

\section{Introduction}

The theory of inequalities has a long history but, from the applicative point of view, it fell
into neglect for hundreds of years because of lack of applications to other branch of mathematics
as well as other sciences, such as physics and engineering. Only in 1934 did Hardy, Littlewood
and P\'{o}lya transformed the field of inequalities from a collection of isolated formulas into
a systematic discipline \cite{H34}.  After that, an enormous amount of effort has been devoted
to the discovery of new types of inequalities and to applications of inequalities \cite{MR4260282}.

It is known that many physical problems in various applications are governed by finite
difference equations. Moreover, discrete inequalities play an important role in the continuing
development of the theory of difference equations. This importance seems to have increased
considerably during the past decades. It has attracted the attention of a large number
of researchers, stimulated new research directions, and influenced various aspects
of difference equations and applications. Among the many types of inequalities,
those associated with the names of Jensen \cite{CD11,D15,D18}, Hilbert \cite{K14,Y15},
Wirtinger \cite{AC98,AP96,FTT55}, Chebyshev \cite{P08,ZA20}, Gronwall--Bellman \cite{F15,Q16}
and Opial \cite{A16,BM15,H18,L68,P87,P91} have deep roots and made a great impact on various
branches of mathematics. The development of discrete inequalities resulted in a renewal
of interest in the field and has attracted interest from more researchers
\cite{B57,C09,CR06,D04,H20,M08,M10,M07,ML03,SA21,SKP19,SK19,SS11}.

More recently, some of classical inequalities have been extended to set-valued functions
by Nikodem et al. \cite{NSS}, \v{S}trboja et al. \cite{S2013}, and Zhang et al. \cite{ZGC21,ZG21},
especially to interval-valued functions by Chalco-Cano et al. \cite{C12,CL15}, Costa et al.
\cite{C17,CF,C19}, Flores-Franuli\v{c} et al. \cite{FC}, Rom\'{a}n-Flores et al. \cite{R18},
and Zhao et al. \cite{ZA18,ZA20,ZA21,ZY19}. The present article is, in some sense, a continuation
of the previous work \cite{ZA21}. Here, we establish some new discrete inequalities of Opial type
involving sequences of intervals and their forward (backward) difference operator. Furthermore, our
present results can be considered as tools for further research in interval difference equations
and inequalities for interval-valued functions, among others.

The paper is organized as follows. Section~\ref{sec:2} contains some necessary preliminaries.
In Section~\ref{sec:3}, we present some new interval Opial type inequalities involving the
backward gH-difference operator, and present some examples to illustrate our theorems.
In Section~\ref{sec:4}, some new discrete Opial type inequalities, involving the forward
gH-difference operator, are given. Finally, in the concluding Section~\ref{sec:5}, 
we summarize our results and outline some possible future work directions.

% -----------------------------------------------------------------

\section{Preliminaries}
\label{sec:2}

We begin by recalling some basic notations, definitions, and results of interval analysis.
We define an interval $u$ by
$$
u=[\underline{u},\overline{u}]
=\{t\in\mathbb{R}|\ \underline{u}\leq t\leq\overline{u}\}.
$$
We write $len(u)=\overline{u}-\underline{u}$.
If $len(u)=0$, then $u$ is called a degenerate interval.
The set of all intervals of $\mathbb{R}$ is denoted by $\mathbb{R}_{\mathcal{I}}$.
For $\lambda\in \mathbb{R}$ and $u\in\mathbb{R}_{\mathcal{I}}$, $\lambda u$ is defined by
\begin{equation*}
\lambda[\underline{u},\overline{u}]=
\begin{cases}
[\lambda\underline{u},\lambda\overline{u}]& \text{if $\lambda\geq0$},\\
[\lambda\overline{u},\lambda\underline{u}]& \text{if $\lambda<0$}.
\end{cases}
\end{equation*}
For $u=[\underline{u},\overline{u}]$ and $v=[\underline{v},\overline{v}]$,
the four arithmetic operators (+,-,$\cdot$,/) are defined by
$$
u+v=[\underline{u}+\underline{v},\overline{u}+\overline{v}],
$$
$$
u-v=[\underline{u}-\overline{v},\overline{u}-\underline{v}],
$$
$$
u\cdot v=\big[\min\{\underline{u}\underline{v},\underline{u}\overline{v},
\overline{u}\underline{v},\overline{u}\overline{v}\},
\max\{\underline{u}\underline{v},\underline{u}\overline{v},
\overline{u}\underline{v},\overline{u}\overline{v}\}\big],
$$
\begin{equation*}
u/v=\big[\min\{\underline{u}/\underline{v},\underline{u}/\overline{v},
\overline{u}/\underline{v},\overline{u}/\overline{v}\},
\max\{\underline{u}/\underline{v},\underline{u}/\overline{v},
\overline{u}/\underline{v},\overline{u}/\overline{v}\}\big],
{\rm where}\ \  0\notin[\underline{v},\overline{v}].
\end{equation*}
Note that $\mathbb{R}_{\mathcal{I}}$ with the above operations
(i.e., the Minkowski addition and the scalar multiplication)
is \emph{not} a linear space since an interval does not have
inverse element and, therefore, the subtraction does not have adequate properties.
For example, when subtracting two intervals $u$ and $v$,
the width of the result is the sum of the widths of $u$ and $v$, i.e.,
$$
len(u-v)=len(u)+len(v).
$$
To partially overcome this situation, Hukuhara \cite{H67}
introduced the following H-difference:
$$
u\ominus v=w\Leftrightarrow u=v+w.
$$
Unfortunately, the H-difference does not always exist for any $u$ and $v$.

In \cite{S10}, Stefanini introduced the gH-difference as follows:
\begin{equation*}
u\bm{\ominus_{g}}v=w
\Leftrightarrow
\begin{cases}
\ \ \ \ (a)\ u=v+w,\\
or \ (b)\ v=u+(-1)w.
\end{cases}
\end{equation*}
The gH-difference always exists for any $u$ and $v$. We also have
$$
u\bm{\ominus_{g}}v=
\big[\min\{\underline{u}-\underline{v}, \overline{u}-\overline{v}\},
\max\{\underline{u}-\underline{v}, \overline{u}-\overline{v}\}\big].
$$

The Hausdorff distance between $u$ and $v$ is defined by
$$
d(u,v)
=\max\Big\{|\underline{u}-\underline{v}|,
|\overline{u}-\overline{v}|\Big\}.
$$
Then, $(\mathbb{R}_{\mathcal{I}}, d)$ is a complete metric space. Note that
$(\mathbb{R}_{\mathcal{I}}, +, \cdot)$ is a quasi-linear space (see \cite{C19})
equipped with the quasi-norm  $\|\cdot\|$, which is given by
$$
\|u\|=d(u,[0,0])=d([\underline{u},\overline{u}],[0,0])
=\max\{|\underline{u}|,|\overline{u}|\}
$$
for all $u\in \mathbb{R}_{\mathcal{I}}$.

On $[a,b]$, $u_{i}$ is called increasing if and only if
$\underline{u_{i}}$ and $\overline{u_{i}}$ are increasing;
$u_{i}$ and $v_{i}$ are synchronous (asynchronous) monotone
if they have the same (opposite) monotonicity;
$u_{i}$ is $\mu$-increasing if $len(u_{i})$ is increasing.
One defines $u^{\lambda}$ by
$$
u^{\lambda}=\{t^{\lambda}|\ t\in[\underline{u},\overline{u}]\}.
$$

For convenience, we now recall the classical Opial's inequality:

\begin{theorem}[continuous Opial inequality \cite{O60}]
\label{thm 2.1}
Let $F\in C^{1}[0,h]$, $F(0)=F(h)=0$ and $F(t)>0$ for $t\in(0,h)$. Then,
\begin{equation}
\label{1.1}
\int^{h}_{0}|F(t)F'(t)|dt
\leq
\frac{h}{4}\int^{h}_{0}\big(F'(t)\big)^{2}dt,
\end{equation}
where $\frac{h}{4}$ is the best possible.
\end{theorem}

A discrete analogue of Theorem~\ref{thm 2.1} is the following:

\begin{theorem}[discrete Opial inequality \cite{AP95}]
\label{thm 2.2}
Let $\{u_{i}\}_{i=0}^{n}$ be a sequence of numbers with $u_{0}=0$ and $u_{n}=0$. Then,
\begin{equation}\label{1.2}
\sum^{n-1}_{i=1}|u_{i}\Delta u_{i}|
\leq
\frac{1}{2}\Big[\frac{n+1}{2}\Big]\sum^{n-1}_{i=0}|\Delta u_{i}|^{2},
\end{equation}
where $\Delta$ is the forward difference operator
and  $[\cdot]$ is the greatest integer function.
\end{theorem}

Many generalizations of Theorem~\ref{thm 2.2} are available in the literature:
see, e.g., \cite{A16,BM15,H18}. In Sections~\ref{sec:3} and \ref{sec:4}, we
give several extensions of Theorem~\ref{thm 2.2} for sequences of intervals.

% -----------------------------------------------------------------

\section{Opial type inequalities involving the backward/nabla gH-difference operator}
\label{sec:3}

\begin{definition}
\label{defn3.1}
Let $\{u_{i}\}$ be a sequence of intervals. We define the forward (delta)
gH-difference operator $\Delta u$ by
$$
\Delta u_{i}=u_{i+1}\bm{\ominus_{g}}u_{i}.
$$
Similarly, we define the backward (nabla) gH-difference operator
$\nabla u$ by
$$
\nabla u_{i}=u_{i}\bm{\ominus_{g}}u_{i-1}.
$$
\end{definition}

\begin{remark}
\label{rmk 3}
Note that if $\{u_{i}\}$ is a sequence of degenerate intervals,
then the forward (backward) gH-difference operator reduces
to the classical forward (backward) difference operator.
\end{remark}

Lemma~\ref{lem 3.1} has been obtained by Lee in \cite{L68}.
Here we give a new and more direct proof.

\begin{lemma}[cf. \cite{L68}]
\label{lem 3.1}
Let $\{u_{i}\}_{i=1}^{n}$ be a non-decreasing sequence of non-negative real numbers,
$u_{0}=0$, and $\lambda_{1},\ \lambda_{2}\geq1$. Then,
\begin{equation}
\label{eq3.1}
\sum^{n}_{i=1}u_{i}^{\lambda_{1}}\big(\nabla u_{i}\big)^{\lambda_{2}}
\leq
\frac{\lambda_{2}(n+1)^{\lambda_{1}}}{\lambda_{1}+\lambda_{2}}
\sum^{n}_{i=1}\big(\nabla u_{i}\big)^{\lambda_{1}+\lambda_{2}}.
\end{equation}
\end{lemma}

\begin{proof}
Since $\nabla u_{i}=u_{i}-u_{i-1}$, we have $u_{i}=\sum_{j=1}^{i}\nabla u_{j}$.
We may rewrite \eqref{eq3.1} as
\begin{equation}
\label{eq3.2}
\sum^{n}_{i=1}\bigg(\sum_{j=1}^{i}\nabla u_{j}\bigg)^{\lambda_{1}}\big(
\nabla u_{i}\big)^{\lambda_{2}}
\leq
\frac{\lambda_{2}(n+1)^{\lambda_{1}}}{\lambda_{1}+\lambda_{2}}
\sum^{n}_{i=1}\big(\nabla u_{i}\big)^{\lambda_{1}+\lambda_{2}}.
\end{equation}
We shall prove \eqref{eq3.2} by induction on $n$.
Clearly, \eqref{eq3.2} holds with $n=1$.
Assume that it holds for $n$, so that
\begin{equation}
\label{eq3.03}
\begin{split}
\sum^{n+1}_{i=1}&\bigg(\sum_{j=1}^{i}\nabla
u_{j}\bigg)^{\lambda_{1}}\big(\nabla u_{i}\big)^{\lambda_{2}}\\
&\leq\frac{\lambda_{2}(n+1)^{\lambda_{1}}}{\lambda_{1}+\lambda_{2}}
\Bigg(\sum^{n}_{i=1}\big(\nabla u_{i}\big)^{\lambda_{1}+\lambda_{2}}
+\frac{\lambda_{1}+\lambda_{2}}{\lambda_{2}}\bigg(\frac{1}{n+1}
\sum_{j=1}^{n+1}\nabla u_{j}\bigg)^{\lambda_{1}}\big(\nabla u_{n+1}\big)^{\lambda_{2}}\Bigg)\\
&\leq\frac{\lambda_{2}(n+1)^{\lambda_{1}}}{\lambda_{1}+\lambda_{2}}
\Bigg(\sum^{n}_{i=1}\big(\nabla u_{i}\big)^{\lambda_{1}+\lambda_{2}}
+\frac{\lambda_{1}+\lambda_{2}}{\lambda_{2}}
A_{n+1}^{\lambda_{1}}\big(\nabla u_{n+1}\big)^{\lambda_{2}}\Bigg),
\end{split}
\end{equation}
where $A_{n+1}=\frac{1}{n+1}\sum_{j=1}^{n+1}\nabla u_{j}$.
Using Young's inequality, we have
$$
A_{n+1}^{\lambda_{1}}\big(\nabla u_{n+1}\big)^{\lambda_{2}}
\leq\frac{\lambda_{1}}{\lambda_{1}+\lambda_{2}}A_{n+1}^{\lambda_{1}+\lambda_{2}}
+\frac{\lambda_{2}}{\lambda_{1}+\lambda_{2}}\big(\nabla u_{n+1}\big)^{\lambda_{1}+\lambda_{2}}.
$$
Then,
\begin{equation}
\label{eq3.04}
\begin{split}
\frac{\lambda_{1}+\lambda_{2}}{\lambda_{2}}A_{n+1}^{\lambda_{1}}\big(
\nabla u_{n+1}\big)^{\lambda_{2}}
\leq\frac{\lambda_{1}}{\lambda_{2}}A_{n+1}^{\lambda_{1}+\lambda_{2}}
+\big(\nabla u_{n+1}\big)^{\lambda_{1}+\lambda_{2}}.
\end{split}
\end{equation}
Thanks to H\"{o}lder's inequality, it follows that
\begin{equation*}
\begin{split}
A_{n+1}&=\frac{1}{n+1}\sum_{j=1}^{n+1}\nabla u_{j}\\
&\leq\Bigg(\sum_{j=1}^{n+1}
\bigg(\frac{1}{n+1}\bigg)^{\frac{\lambda_{1}
+\lambda_{2}}{\lambda_{1}+\lambda_{2}-1}}\Bigg)^{\frac{\lambda_{1}
+\lambda_{2}-1}{\lambda_{1}+\lambda_{2}}}
\Bigg(\sum_{j=1}^{n+1}\big(\nabla
u_{i}\big)^{\lambda_{1}+\lambda_{2}}\Bigg)^{\frac{1}{\lambda_{1}+\lambda_{2}}}\\
&\leq\bigg(\frac{1}{n+1}\bigg)^{\frac{1}{\lambda_{1}+\lambda_{2}}}
\Bigg(\sum_{j=1}^{n+1}\big(\nabla u_{i}\big)^{\lambda_{1}
+\lambda_{2}}\Bigg)^{\frac{1}{\lambda_{1}+\lambda_{2}}}.
\end{split}
\end{equation*}
Consequently, we get
\begin{equation}
\label{eq3.05}
\begin{split}
A_{n+1}^{\lambda_{1}+\lambda_{2}}
\leq\frac{1}{n+1}\sum_{j=1}^{n+1}\big(\nabla u_{i}\big)^{\lambda_{1}+\lambda_{2}}.
\end{split}
\end{equation}
Thus, combining \eqref{eq3.03}, \eqref{eq3.04} and \eqref{eq3.05}, we have
\begin{equation*}
\begin{split}
\sum^{n+1}_{i=1}\bigg(\sum_{j=1}^{i}\nabla u_{j}\bigg)^{\lambda_{1}}
\big(\nabla u_{i}\big)^{\lambda_{2}}
&\leq\frac{\lambda_{2}(n+1)^{\lambda_{1}}}{\lambda_{1}+\lambda_{2}}
\Bigg(\sum^{n}_{i=1}\big(\nabla u_{i}\big)^{\lambda_{1}+\lambda_{2}}
+\frac{\lambda_{1}+\lambda_{2}}{\lambda_{2}}A_{n+1}^{\lambda_{1}}
\big(\nabla u_{n+1}\big)^{\lambda_{2}}\Bigg)\\
&\leq\frac{\lambda_{2}(n+1)^{\lambda_{1}}}{\lambda_{1}+\lambda_{2}}
\Bigg(\sum^{n}_{i=1}\big(\nabla u_{i}\big)^{\lambda_{1}+\lambda_{2}}
+\frac{\lambda_{1}}{\lambda_{2}}A_{n+1}^{\lambda_{1}+\lambda_{2}}
+\big(\nabla u_{n+1}\big)^{\lambda_{1}+\lambda_{2}}\Bigg)\\
&\leq\frac{\lambda_{2}(n+1)^{\lambda_{1}}}{\lambda_{1}+\lambda_{2}}
\Bigg(\sum^{n+1}_{i=1}\big(\nabla u_{i}\big)^{\lambda_{1}+\lambda_{2}}
+\frac{\lambda_{1}}{n+1}\sum_{j=1}^{n+1}\big(\nabla u_{i}
\big)^{\lambda_{1}+\lambda_{2}}\Bigg)\\
&=\frac{\lambda_{2}(n+1)^{\lambda_{1}}}{\lambda_{1}+\lambda_{2}}
\cdot\frac{n+1+\lambda_{1}}{n+1}\cdot\sum_{j=1}^{n+1}
\big(\nabla u_{i}\big)^{\lambda_{1}+\lambda_{2}}\\
&=\frac{\lambda_{2}\Big[(n+1)^{\lambda_{1}}+\lambda_{1}(n+1)^{\lambda_{1}-1}\Big]}{\lambda_{1}
+\lambda_{2}}\sum_{j=1}^{n+1}\big(\nabla u_{i}\big)^{\lambda_{1}+\lambda_{2}}\\
&\leq\frac{\lambda_{2}(n+2)^{\lambda_{1}}}{\lambda_{1}+\lambda_{2}}\sum_{j=1}^{n+1}
\big(\nabla u_{i}\big)^{\lambda_{1}+\lambda_{2}}.
\end{split}
\end{equation*}
The proof is complete.
\end{proof}

Thanks to Lemma~\ref{lem 3.1}, we can easily obtain the following Lemma~\ref{lem 3.01},

\begin{lemma}
\label{lem 3.01}
Let $\{u_{i}\}_{i=1}^{n}$ be a sequence of numbers,
$u_{0}=0$, and $\lambda_{1},\ \lambda_{2}\geq1$. Then,
\begin{equation}
\label{eq3.01}
\sum^{n}_{i=1}|u_{i}|^{\lambda_{1}}\big|\nabla u_{i}\big|^{\lambda_{2}}
\leq
\frac{\lambda_{2}(n+1)^{\lambda_{1}}}{\lambda_{1}+\lambda_{2}}
\sum^{n}_{i=1}\big|\nabla u_{i}\big|^{\lambda_{1}+\lambda_{2}}.
\end{equation}
\end{lemma}

\begin{proof}
Since $|\nabla u_{i}|=|u_{i}-u_{i-1}|$, we have $|u_{i}|\leq\sum_{j=1}^{i}|\nabla u_{j}|$.
The following proof is similar to that of  Lemma~\ref{lem 3.1} and is
omitted here.
\end{proof}

We are now ready to formulate and prove our first original result.

\begin{theorem}
\label{thm 3.1}
Let $\{u_{i}\}_{i=1}^{n}$ be a sequence of intervals, $u_{0}=[0,0]$,
and $\lambda_{1},\lambda_{2}\geq1$.
If $u_{i}$ is monotone and $\mu$-increasing, then
\begin{equation}
\label{3.1}
\sum^{n}_{i=1}\Big\|u_{i}^{\lambda_{1}}\big(\nabla u_{i}\big)^{\lambda_{2}}\Big\|
\leq
\frac{\lambda_{2}(n+1)^{\lambda_{1}}}{\lambda_{1}+\lambda_{2}}
\sum^{n}_{i=1}\big\|\nabla u_{i}\big\|^{\lambda_{1}+\lambda_{2}}.
\end{equation}
\end{theorem}

\begin{proof}
Suppose that $u_{i}$ is increasing and $\mu$-increasing. Then,
$$
u_{i}^{\lambda_{1}}
=\Big[\underline{u_{i}}^{\lambda_{1}},
\overline{u_{i}}^{\lambda_{1}}\Big],\ \
\big(\nabla u_{i}\big)^{\lambda_{2}}
=\Big[\big(\underline{\nabla u_{i}}\big)^{\lambda_{2}},
\big(\overline{\nabla u_{i}}\big)^{\lambda_{2}}\Big].
$$
Consequently, we obtain that
$$
u_{i}^{\lambda_{1}}\big(\nabla u_{i}\big)^{\lambda_{2}}
=\Big[\underline{u_{i}}^{\lambda_{1}}\big(\underline{\nabla u_{i}}\big)^{\lambda_{2}},
\overline{u_{i}}^{\lambda_{1}}\big(\overline{\nabla u_{i}}\big)^{\lambda_{2}}\Big].
$$
If $u_{i}$ is decreasing and $\mu$-increasing, then
\begin{equation*}
u_{i}^{\lambda_{1}}=
\begin{cases}
\Big[\underline{u_{i}}^{\lambda_{1}},
\overline{u_{i}}^{\lambda_{1}}\Big]&
\text{if $\lambda_{1}$ is odd},\\[2mm]
\Big[\overline{u_{i}}^{\lambda_{1}},
\underline{u_{i}}^{\lambda_{1}}\Big]&
\text{if $\lambda_{1}$ is even},
\end{cases}
\end{equation*}
\begin{equation*}
\big(\nabla u_{i}\big)^{\lambda_{2}}=
\begin{cases}
\Big[\big(\underline{\nabla u_{i}}\big)^{\lambda_{2}},
\big(\overline{\nabla u_{i}}\big)^{\lambda_{2}}\Big]
& \text{if $\lambda_{2}$ is odd},\\[2mm]
\Big[\big(\overline{\nabla u_{i}}\big)^{\lambda_{2}},
\big(\underline{\nabla u_{i}}\big)^{\lambda_{2}}\Big]
& \text{if $\lambda_{2}$ is even}.
\end{cases}
\end{equation*}
Consequently, we obtain
\begin{equation*}
u_{i}^{\lambda_{1}}\big(\nabla u_{i}\big)^{\lambda_{2}}
=
\begin{cases}
\Big[\overline{u_{i}}^{\lambda_{1}}\big(\overline{\nabla u_{i}}\big)^{\lambda_{2}},
\underline{u_{i}}^{\lambda_{1}}\big(\underline{\nabla u_{i}}\big)^{\lambda_{2}}\Big]
& \text{if $\lambda_{1}$ and $\lambda_{2}$ are odd},\\[2mm]
\Big[\overline{u_{i}}^{\lambda_{1}}\big(\overline{\nabla u_{i}}\big)^{\lambda_{2}},
\underline{u_{i}}^{\lambda_{1}}\big(\underline{\nabla u_{i}}\big)^{\lambda_{2}}\Big]
& \text{if $\lambda_{1}$ and $\lambda_{2}$ are even},\\[2mm]
\Big[\underline{u_{i}}^{\lambda_{1}}\big(\underline{\nabla u_{i}}\big)^{\lambda_{2}},
\overline{u_{i}}^{\lambda_{1}}\big(\overline{\nabla u_{i}}\big)^{\lambda_{2}}\Big]
& \text{if $\lambda_{1}$ is odd and $\lambda_{2}$ is even},\\[2mm]
\Big[\underline{u_{i}}^{\lambda_{1}}\big(\underline{\nabla u_{i}}\big)^{\lambda_{2}},
\overline{u_{i}}^{\lambda_{1}}\big(\overline{\nabla u_{i}}\big)^{\lambda_{2}}\Big]
& \text{if $\lambda_{1}$ is even and $\lambda_{2}$ is odd}.
\end{cases}
\end{equation*}
By Lemma~\ref{lem 3.01}, it follows that
\begin{align*}
\sum^{n}_{i=1}\Big\|u_{i}^{\lambda_{1}}\big(\nabla u_{i}\big)^{\lambda_{2}}\Big\|
&=\sum^{n}_{i=1}\Big\|\Big[
\min\big\{\underline{u_{i}}^{\lambda_{1}}\big(\underline{\nabla u_{i}}\big)^{\lambda_{2}},
\overline{u_{i}}^{\lambda_{1}}\big(\overline{\nabla u_{i}}\big)^{\lambda_{2}}\big\},
\max\big\{\underline{u_{i}}^{\lambda_{1}}\big(\underline{\nabla u_{i}}\big)^{\lambda_{2}},
\overline{u_{i}}^{\lambda_{1}}\big(\overline{\nabla u_{i}}\big)^{\lambda_{2}}\big\}
\Big]\Big\|\\
&=\sum^{n}_{i=1}\max\Big\{\Big|\underline{u_{i}}^{\lambda_{1}}
\big(\underline{\nabla u_{i}}\big)^{\lambda_{2}}\Big|,
\Big|\overline{u_{i}}^{\lambda_{1}}
\big(\overline{\nabla u_{i}}\big)^{\lambda_{2}}\Big|\Big\}\\
&=\max\Bigg\{\sum^{n}_{i=1}
\Big|\underline{u_{i}}^{\lambda_{1}}\big(\underline{\nabla u_{i}}\big)^{\lambda_{2}}\Big|,
\sum^{n}_{i=1}\Big|\overline{u_{i}}^{\lambda_{1}}
\big(\overline{\nabla u_{i}}\big)^{\lambda_{2}}\Big|\Bigg\}\\
&\leq\frac{\lambda_{2}(n+1)^{\lambda_{1}}}{\lambda_{1}+\lambda_{2}}
\max\Bigg\{\sum^{n}_{i=1}
\Big|\underline{\nabla u_{i}}^{\lambda_{1}+\lambda_{2}}\Big|,
\sum^{n}_{i=1}\Big|\overline{\nabla u_{i}}^{\lambda_{1}+\lambda_{2}}\Big|\Bigg\}\\
&\leq\frac{\lambda_{2}(n+1)^{\lambda_{1}}}{\lambda_{1}+\lambda_{2}}
\sum^{n}_{i=1}\max\Big\{\big|\underline{\nabla u_{i}}\big|^{\lambda_{1}+\lambda_{2}},
\big|\overline{\nabla u_{i}}\big|^{\lambda_{1}+\lambda_{2}}\Big\}\\
&\leq\frac{\lambda_{2}(n+1)^{\lambda_{1}}}{\lambda_{1}+\lambda_{2}} \sum^{n}_{i=1}\|
\nabla u_{i}\|^{\lambda_{1}+\lambda_{2}}.
\end{align*}
This concludes the proof.
\end{proof}

Follows an example of application of our Theorem~\ref{thm 3.1}.

\begin{example}
\label{ex 3.1}
Suppose that
$\{u_{i}\}_{i=0}^{n}=\{[0,0],[1,2],[2,4],\ldots,[n,2n]\}$
and $\lambda_{1},\ \lambda_{2}\geq1$. By Theorem~\ref{thm 3.1}, we have
\begin{equation*}
\begin{split}
\sum^{n}_{i=1}\Big\|u_{i}^{\lambda_{1}}\big(\nabla u_{i}\big)^{\lambda_{2}}\Big\|
&=\sum^{n}_{i=1}\Big\|\big[i^{\lambda_{1}},(2i)^{\lambda_{1}}\big]\cdot [1,2]^{\lambda_{2}}\Big\|\\
&=2^{\lambda_{1}+\lambda_{2}}\sum^{n}_{i=1}i^{\lambda_{1}}\\
&\leq\frac{\lambda_{2}n(n+1)^{\lambda_{1}}}{\lambda_{1}+\lambda_{2}}2^{\lambda_{1}+\lambda_{2}}\\
&=\frac{\lambda_{2}(n+1)^{\lambda_{1}}}{\lambda_{1}+\lambda_{2}}
\sum^{n}_{i=1}2^{\lambda_{1}+\lambda_{2}}\\
&=\frac{\lambda_{2}(n+1)^{\lambda_{1}}}{\lambda_{1}+\lambda_{2}}
\sum^{n}_{i=1}\big\|\nabla u_{i}\big\|^{\lambda_{1}+\lambda_{2}}.
\end{split}
\end{equation*}
\end{example}

\begin{lemma}
\label{lem 3.02}
Let $\{u_{i}\}_{i=1}^{m}$ be a sequence of numbers,
$u_{m}=0$, and $\lambda_{1},\ \lambda_{2}\geq1$. Then,
\begin{equation}
\label{eq3.02}
\sum^{m-1}_{i=n}|u_{i}|^{\lambda_{1}}\big|\nabla u_{i}\big|^{\lambda_{2}}
\leq
\frac{\lambda_{2}(m-n+1)^{\lambda_{1}}}{\lambda_{1}+\lambda_{2}}
\sum^{m}_{i=n}\big|\nabla u_{i}\big|^{\lambda_{1}+\lambda_{2}}.
\end{equation}
\end{lemma}

\begin{proof}
Since $|\nabla u_{i}|=|u_{i}-u_{i-1}|$, we have $|u_{i}|\leq\sum_{j=i+1}^{m}|\nabla u_{j}|$.
The following proof is similar to that of  Lemma~\ref{lem 3.1} and is
omitted here.
\end{proof}

Similarly to Theorem~\ref{thm 3.1}, we obtain an analogous result
when $u_{i}$ is monotone but $\mu$-decreasing instead of $\mu$-increasing.

\begin{theorem}
\label{thm 3.2}
Let $\{u_{i}\}_{i=0}^{m}$ be a sequence of intervals, $u_{m}=[0,0]$,
and $\lambda_{1}, \lambda_{2}\geq1$. If $u_{i}$ is monotone and $\mu$-decreasing, then
\begin{equation}
\sum^{m-1}_{i=n}\Big\|u_{i}^{\lambda_{1}}\big(\nabla u_{i}\big)^{\lambda_{2}}\Big\|
\leq
\frac{\lambda_{2}(m-n+1)^{\lambda_{1}}}{\lambda_{1}+\lambda_{2}}
\sum^{m}_{i=n}\big\|\nabla u_{i}\big\|^{\lambda_{1}+\lambda_{2}}.
\end{equation}
\end{theorem}

\begin{proof}
Suppose that $u_{i}$ is increasing and $\mu$-decreasing. Then,
\begin{equation*}
u_{i}^{\lambda_{1}}=
\begin{cases}
\Big[\underline{u_{i}}^{\lambda_{1}},
\overline{u_{i}}^{\lambda_{1}}\Big]&
\text{if $\lambda_{1}$ is odd},\\[2mm]
\Big[\overline{u_{i}}^{\lambda_{1}},
\underline{u_{i}}^{\lambda_{1}}\Big]&
\text{if $\lambda_{1}$ is even},
\end{cases}
\end{equation*}
and
$$
\big(\nabla u_{i}\big)^{\lambda_{2}}
=\Big[\big(\overline{\nabla u_{i}}\big)^{\lambda_{2}},
\big(\underline{\nabla u_{i}}\big)^{\lambda_{2}}\Big].
$$
Consequently, we obtain
\begin{equation}
\label{eq 3.8}
u_{i}^{\lambda_{1}}\big(\nabla u_{i}\big)^{\lambda_{2}}
=
\begin{cases}
\Big[\underline{u_{i}}^{\lambda_{1}}\big(\underline{\nabla u_{i}}\big)^{\lambda_{2}},
\overline{u_{i}}^{\lambda_{1}}\big(\overline{\nabla u_{i}}\big)^{\lambda_{2}}\Big]
& \text{if $\lambda_{1}$ is odd},\\[2mm]
\Big[\overline{u_{i}}^{\lambda_{1}}\big(\overline{\nabla u_{i}}\big)^{\lambda_{2}},
\underline{u_{i}}^{\lambda_{1}}\big(\underline{\nabla u_{i}}\big)^{\lambda_{2}}\Big]
& \text{if $\lambda_{1}$ is even}.\\[2mm]
\end{cases}
\end{equation}

If $u_{i}$ is decreasing and $\mu$-decreasing, then
$$
u_{i}^{\lambda_{1}}=
\Big[\underline{u_{i}}^{\lambda_{1}},
\overline{u_{i}}^{\lambda_{1}}\Big],
$$
\begin{equation*}
\big(\nabla u_{i}\big)^{\lambda_{2}}=
\begin{cases}
\Big[\big(\overline{\nabla u_{i}}\big)^{\lambda_{2}},
\big(\underline{\nabla u_{i}}\big)^{\lambda_{2}}\Big]
& \text{if $\lambda_{2}$ is odd},\\[2mm]
\Big[\big(\underline{\nabla u_{i}}\big)^{\lambda_{2}},
\big(\overline{\nabla u_{i}}\big)^{\lambda_{2}}\Big]
& \text{if $\lambda_{2}$ is even}.

\end{cases}
\end{equation*}
Consequently, we obtain
\begin{equation}
\label{eq 3.9}
u_{i}^{\lambda_{1}}\big(\nabla u_{i}\big)^{\lambda_{2}}
=
\begin{cases}
\Big[\overline{u_{i}}^{\lambda_{1}}\big(\overline{\nabla u_{i}}\big)^{\lambda_{2}},
\underline{u_{i}}^{\lambda_{1}}\big(\underline{\nabla u_{i}}\big)^{\lambda_{2}}\Big]
& \text{if $\lambda_{2}$ is odd},\\[2mm]
\Big[\underline{u_{i}}^{\lambda_{1}}\big(\underline{\nabla u_{i}}\big)^{\lambda_{2}},
\overline{u_{i}}^{\lambda_{1}}\big(\overline{\nabla u_{i}}\big)^{\lambda_{2}}\Big]
& \text{if $\lambda_{2}$ is even}.\\[2mm]
\end{cases}
\end{equation}
By \ref{eq 3.8}, \ref{eq 3.9} and Lemma~\ref{lem 3.02}, it follows that
\begin{align*}
\sum^{m-1}_{i=n}\Big\|u_{i}^{\lambda_{1}}\big(\nabla u_{i}\big)^{\lambda_{2}}\Big\|
&=\sum^{m-1}_{i=n}\Big\|\Big[
\min\big\{\underline{u_{i}}^{\lambda_{1}}\big(\underline{\nabla u_{i}}\big)^{\lambda_{2}},
\overline{u_{i}}^{\lambda_{1}}\big(\overline{\nabla u_{i}}\big)^{\lambda_{2}}\big\},
\max\big\{\underline{u_{i}}^{\lambda_{1}}\big(\underline{\nabla u_{i}}\big)^{\lambda_{2}},
\overline{u_{i}}^{\lambda_{1}}\big(\overline{\nabla u_{i}}\big)^{\lambda_{2}}\big\}
\Big]\Big\|\\
&=\sum^{m-1}_{i=n}\max\Big\{\Big|\underline{u_{i}}^{\lambda_{1}}
\big(\underline{\nabla u_{i}}\big)^{\lambda_{2}}\Big|,
\Big|\overline{u_{i}}^{\lambda_{1}}
\big(\overline{\nabla u_{i}}\big)^{\lambda_{2}}\Big|\Big\}\\
&=\max\Bigg\{\sum^{m-1}_{i=n}
\Big|\underline{u_{i}}^{\lambda_{1}}\big(\underline{\nabla u_{i}}\big)^{\lambda_{2}}\Big|,
\sum^{m-1}_{i=n}\Big|\overline{u_{i}}^{\lambda_{1}}
\big(\overline{\nabla u_{i}}\big)^{\lambda_{2}}\Big|\Bigg\}\\
&\leq\frac{\lambda_{2}(m-n+1)^{\lambda_{1}}}{\lambda_{1}+\lambda_{2}}
\max\Bigg\{\sum^{m}_{i=n}
\Big|\underline{\nabla u_{i}}\Big|^{\lambda_{1}+\lambda_{2}},
\sum^{m}_{i=n}\Big|\overline{\nabla u_{i}}\Big|^{\lambda_{1}+\lambda_{2}}\Bigg\}\\
&\leq\frac{\lambda_{2}(m-n+1)^{\lambda_{1}}}{\lambda_{1}+\lambda_{2}}
\sum^{m}_{i=n}\max\Big\{\big|\underline{\nabla u_{i}}\big|^{\lambda_{1}+\lambda_{2}},
\big|\overline{\nabla u_{i}}\big|^{\lambda_{1}+\lambda_{2}}\Big\}\\
&\leq\frac{\lambda_{2}(m-n+1)^{\lambda_{1}}}{\lambda_{1}+\lambda_{2}} \sum^{m}_{i=n}\|
\nabla u_{i}\|^{\lambda_{1}+\lambda_{2}}.
\end{align*}
This concludes the proof.
\end{proof}

\begin{example}
\label{ex 3.2}
Suppose that
$$
\{u_{i}\}_{i=1}^{n}=\bigg\{[1,2],\Big[\frac{1}{2},1\Big],\Big[\frac{1}{i},\frac{2}{i}\Big],
\ldots,\Big[\frac{1}{n-1},\frac{2}{n-1}\Big],[0,0]\bigg\}
$$
and $\lambda_{1}=1$ and $\lambda_{2}=2$. By induction on $n$, we have
\begin{equation*}
\begin{split}
\sum^{n-1}_{i=2}\Big\|u_{i}^{\lambda_{1}}\big(\nabla u_{i}\big)^{\lambda_{2}}\Big\|
&=\sum^{n-1}_{i=2}\Big\|\Big[\frac{1}{i},\frac{2}{i}\Big]
\cdot \Big[\frac{-2}{i(i-1)},\frac{-1}{i(i-1)}\Big]^{2}\Big\|\\
&=\sum^{n-1}_{i=2}\Big\|\Big[\frac{1}{i},\frac{2}{i}\Big]
\cdot \Big[\frac{1}{i^{2}(i-1)^{2}},\frac{4}{i^{2}(i-1)^{2}}\Big]\Big\|\\
&=\sum^{n-1}_{i=2}\frac{8}{i^{3}(i-1)^{2}}\\
&\leq\frac{2(n-1)}{3}\sum^{n}_{i=2}\frac{2^{3}}{i^{3}(i-1)^{3}}\\
&=\frac{\lambda_{2}(n-1)^{\lambda_{1}}}{\lambda_{1}+\lambda_{2}}
\sum^{n}_{i=2}\big\|\nabla u_{i}\big\|^{\lambda_{1}+\lambda_{2}}.
\end{split}
\end{equation*}
\end{example}

Theorem~\ref{thm 3.1} is a special case of our next Theorem~\ref{thm 3.3}.

\begin{theorem}
\label{thm 3.3}
Let $\{u_{i}\}_{i=1}^{n}$ be a sequence of intervals, $u_{0}=[0,0]$,
and $\lambda_{1}, \lambda_{2}\geq1$. If $\{u_{i}\}_{i=1}^{n}$ is
piecewise alternate monotone, piecewise alternate $\mu$-monotone,
and there is no other point $i$ such that $u_{i}=[0,0]$, then
\begin{equation}
\label{3.2}
\sum^{n}_{i=1}\Big\|u_{i}^{\lambda_{1}}\big(\nabla u_{i}\big)^{\lambda_{2}}\Big\|
\leq
\frac{\lambda_{2}(n+1)^{\lambda_{1}}}{\lambda_{1}+\lambda_{2}}
\sum^{n}_{i=1}\big\|\nabla u_{i}\big\|^{\lambda_{1}+\lambda_{2}}.
\end{equation}
\end{theorem}

\begin{proof}
First, suppose that there exists a finite number of points such that
$$
1=i_{0}\leq i_{1}<i_{2}<\cdots <i_{k-1}<i_{k}=n
$$
and $u_{i}$ is piecewise alternate monotone and piecewise alternate $\mu$-monotone.
By Lemma~\ref{lem 3.1}, we have
\begin{align*}
\sum^{n}_{i=1}\Big\|u_{i}^{\lambda_{1}}\big(\nabla u_{i}\big)^{\lambda_{2}}\Big\|
&=\sum^{i_{1}}_{i=i_{0}}\Big\|u_{i}^{\lambda_{1}}\big(\nabla u_{i}\big)^{\lambda_{2}}\Big\|
+\sum^{i_{2}}_{i=i_{1}+1}\Big\|u_{i}^{\lambda_{1}}\big(\nabla u_{i}\big)^{\lambda_{2}}\Big\|
+\cdots +\sum^{i_{k}}_{i=i_{k-1}+1}\Big\|u_{i}^{\lambda_{1}}\big(\nabla u_{i}\big)^{\lambda_{2}}\Big\|\\
&=\sum_{j=0}^{k-1}\sum^{i_{j+1}}_{i=i_{j}}
\max\Big\{\Big|\underline{u_{i}}^{\lambda_{1}}
\big(\underline{\nabla u_{i}}\big)^{\lambda_{2}}\Big|,
\Big|\overline{u_{i}}^{\lambda_{1}}
\big(\overline{\nabla u_{i}}\big)^{\lambda_{2}}\Big|\Big\}\\
&=\sum^{n}_{i=1}\max\Big\{\Big|\underline{u_{i}}^{\lambda_{1}}
\big(\underline{\nabla u_{i}}\big)^{\lambda_{2}}\Big|,
\Big|\overline{u_{i}}^{\lambda_{1}}
\big(\overline{\nabla u_{i}}\big)^{\lambda_{2}}\Big|\Big\}\\
&\leq\frac{\lambda_{2}(n+1)^{\lambda_{1}}}{\lambda_{1}+\lambda_{2}}
\sum^{n}_{i=1}\big\|\nabla u_{i}\big\|^{\lambda_{1}+\lambda_{2}}.
\end{align*}
The proof is complete.
\end{proof}

Similarly, we can also generalize Theorem~\ref{thm 3.2} as follows.

\begin{theorem}
\label{thm 3.4}
Let $\{u_{i}\}_{i=n}^{m}$ be a sequence of intervals, $u_{m}=[0,0]$,
and $\lambda_{1}, \lambda_{2}\geq1$. If $\{u_{i}\}_{i=0}^{m}$ is
piecewise alternate monotone, piecewise alternate $\mu$-monotone,
and there is no other point $i$ such that $u_{i}=[0,0]$, then
\begin{equation}
\sum^{m-1}_{i=n}\Big\|u_{i}^{\lambda_{1}}\big(\nabla u_{i}\big)^{\lambda_{2}}\Big\|
\leq
\frac{\lambda_{2}(m-n+1)^{\lambda_{1}}}{\lambda_{1}+\lambda_{2}}
\sum^{m}_{i=n}\big\|\nabla u_{i}\big\|^{\lambda_{1}+\lambda_{2}}.
\end{equation}
\end{theorem}

\begin{proof}
The proof is analogous to the one of Theorem~\ref{thm 3.3}.
\end{proof}

As an application of Theorems~\ref{thm 3.3} and \ref{thm 3.4},
we now obtain the following result.

\begin{theorem}
\label{thm 3.5}
Let $\{u_{i}\}_{i=0}^{m}$ be a sequence of intervals, $u_{0}=u_{m}=[0,0]$,
and $\lambda_{1}, \lambda_{2}\geq1$. If $\{u_{i}\}_{i=0}^{m}$ is piecewise
alternate monotone, piecewise alternate $\mu$-monotone, and there is no
other point $i$ such that $u_{i}=[0,0]$, then
\begin{equation}
\label{eq3.5}
\sum^{m-1}_{i=1}\Big\|u_{i}^{\lambda_{1}}\big(\nabla u_{i}\big)^{\lambda_{2}}\Big\|
\leq
\frac{\lambda_{2}\Big(\big[\frac{m}{2}\big]+1\Big)^{\lambda_{1}}}{\lambda_{1}+\lambda_{2}}
\sum^{m}_{i=1}\big\|\nabla u_{i}\big\|^{\lambda_{1}+\lambda_{2}}.
\end{equation}
\end{theorem}

\begin{proof}
Let us take $n=\big[\frac{m}{2}\big]$. By Theorem~\ref{thm 3.3}, we have
\begin{equation}
\label{eq:inq01}
\sum^{[\frac{m}{2}]}_{i=1}\Big\|u_{i}^{\lambda_{1}}\big(\nabla u_{i}\big)^{\lambda_{2}}\Big\|
\leq
\frac{\lambda_{2}\Big(\big[\frac{m}{2}\big]+1\Big)^{\lambda_{1}}}{\lambda_{1}+\lambda_{2}}
\sum^{[\frac{m}{2}]}_{i=1}\big\|\nabla u_{i}\big\|^{\lambda_{1}+\lambda_{2}}.
\end{equation}
Similarly, by Theorem~\ref{thm 3.4}, we have
\begin{equation}
\label{eq:inq02}
\begin{aligned}
\sum^{m-1}_{i=[\frac{m}{2}]+1}\Big\|u_{i}^{\lambda_{1}}\big(\nabla u_{i}\big)^{\lambda_{2}}\Big\|
&\leq \frac{\lambda_{2}\Big(m-\big[\frac{m}{2}\big]\Big)^{\lambda_{1}}}{\lambda_{1}+\lambda_{2}}
\sum^{m}_{i=[\frac{m}{2}]+1}\big\|\nabla u_{i}\big\|^{\lambda_{1}+\lambda_{2}}\\
&\leq\frac{\lambda_{2}\Big(\big[\frac{m}{2}\big]+1\Big)^{\lambda_{1}}}{\lambda_{1}+\lambda_{2}}
\sum^{m}_{i=[\frac{m}{2}]+1}\big\|\nabla u_{i}\big\|^{\lambda_{1}+\lambda_{2}}.
\end{aligned}
\end{equation}
The intended relation (\ref{eq3.5}) follows by adding the above
two inequalities \eqref{eq:inq01} and \eqref{eq:inq02}.
\end{proof}

\begin{example}
\label{ex 3.3}
Suppose that
$$\{u_{i}\}_{i=0}^{5}=\Big\{[0,0],[1,2],[2,4],[3,6],[1,2],[0,0]\Big\},
\lambda_{1}=2, \lambda_{2}=3.$$
Then, we have
$$
\sum^{4}_{i=1}\Big\|u_{i}^{2}\big(\nabla u_{i}\big)^{3}\Big\|
=704
<6048
=\frac{3\cdot3^{2}}{5}\sum^{4}_{i=1}\big\|\nabla u_{i}\big\|^{5}.
$$
Let
$\lambda_{1}=1, \lambda_{2}=2.$
Then, we have
$$
\sum^{4}_{i=1}\Big\|u_{i}\big(\nabla u_{i}\big)^{2}\Big\|
=80
<184
=\frac{2\cdot3}{3}\sum^{4}_{i=1}\big\|\nabla u_{i}\big\|^{3}.
$$
\end{example}

Now, we give new discrete Opial inequalities involving two interval sequences.

\begin{theorem}
\label{thm 3.6}
Let $\{u_{i}\}_{i=0}^{n}$ and $\{v_{i}\}_{i=0}^{n}$ be two sequences of intervals,
$u_{0}=v_{0}=[0,0]$. If $u_{i}$ and $v_{i}$ are synchronous monotone and $\mu$-increasing,
then
\begin{equation}
\label{eq3.6}
\sum^{n}_{i=1}\big\|u_{i-1}\nabla v_{i} + v_{i}\nabla u_{i}\big\|
\leq
\frac{n}{2}\sum^{n}_{i=1}\big\|(\nabla u_{i})^{2} + (\nabla v_{i})^{2}\big\|.
\end{equation}
\end{theorem}

\begin{proof}
Suppose that $u_{i}$ and $v_{i}$ are increasing and $\mu$-increasing. Then,
$u_{i}v_{i}$ is also increasing and $\mu$-increasing. Consequently, we obtain that
\begin{equation*}
\begin{split}
u_{i-1}\nabla v_{i} + v_{i}\nabla u_{i}
&=\Big[\underline{u_{i-1}},\overline{u_{i-1}}\Big]
\cdot \Big[\underline{v_{i}}-\underline{v_{i-1}},\overline{v_{i}}-\overline{v_{i-1}}\Big]
+\Big[\underline{v_{i}},\overline{v_{i}}\Big]
\cdot \Big[\underline{u_{i}}-\underline{u_{i-1}},\overline{u_{i}}-\overline{u_{i-1}}\Big]\\
&=\big[\underline{u_{i-1}}(\underline{v_{i}}-\underline{v_{i-1}})
+\underline{v_{i}}(\underline{u_{i}}-\underline{u_{i-1}}),\
\overline{u_{i-1}}(\overline{v_{i}}-\overline{v_{i-1}})
+\overline{v_{i}}(\overline{u_{i}}-\overline{u_{i-1}})\Big]\\
&=\Big[\underline{u_{i}} \cdot\underline{v_{i}}
-\underline{u_{i-1}}\cdot\underline{v_{i-1}},\
\overline{u_{i}}\cdot\overline{v_{i}}
-\overline{u_{i-1}}\cdot\overline{v_{i-1}}\Big]\\
&=\nabla (u_{i}v_{i}).
\end{split}
\end{equation*}
If $u_{i}$ and $v_{i}$ are decreasing and $\mu$-increasing, then
$u_{i}v_{i}$ is increasing and $\mu$-increasing. Consequently, we obtain that
\begin{align*}
u_{i-1}\nabla v_{i} + v_{i}\nabla u_{i}
&=\Big[\underline{u_{i-1}},\overline{u_{i-1}}\Big]
\cdot \Big[\underline{v_{i}}-\underline{v_{i-1}},\overline{v_{i}}-\overline{v_{i-1}}\Big]
+\Big[\underline{v_{i}},\overline{v_{i}}\Big]
\cdot \Big[\underline{u_{i}}-\underline{u_{i-1}},\overline{u_{i}}-\overline{u_{i-1}}\Big]\\
&=\Big[\overline{u_{i-1}}(\overline{v_{i}}-\overline{v_{i-1}})
+\overline{v_{i}}(\overline{u_{i}}-\overline{u_{i-1}}),\
\underline{u_{i-1}}(\underline{v_{i}}-\underline{v_{i-1}})
+\underline{v_{i}}(\underline{u_{i}}-\underline{u_{i-1}})\Big]\\
&=\Big[\overline{u_{i}}\cdot\overline{v_{i}}
-\overline{u_{i-1}}\cdot\overline{v_{i-1}},\
\underline{u_{i}} \cdot\underline{v_{i}}
-\underline{u_{i-1}}\cdot\underline{v_{i-1}}\Big]\\
&=\nabla (u_{i}v_{i}).
\end{align*}
Then, by the Cauchy--Schwarz inequality, we have
\begin{equation*}
\begin{split}
\sum^{n}_{i=1}\big\|u_{i-1}\nabla v_{i} + v_{i}\nabla u_{i}\big\|
&=\sum^{n}_{i=1}\max\big\{\underline{u_{i}} \cdot\underline{v_{i}}
-\underline{u_{i-1}}\cdot\underline{v_{i-1}},\
\overline{u_{i}}\cdot\overline{v_{i}}
-\overline{u_{i-1}}\cdot\overline{v_{i-1}}\big\}\\
&\leq\max\Bigg\{\sum^{n}_{i=1}\Big(\underline{u_{i}} \cdot\underline{v_{i}}
-\underline{u_{i-1}}\cdot\underline{v_{i-1}}\Big),\
\sum^{n}_{i=1}\Big(\overline{u_{i}}\cdot\overline{v_{i}}
-\overline{u_{i-1}}\cdot\overline{v_{i-1}}\Big)\Bigg\}\\
&\leq\max\big\{\underline{u_{n}}\cdot\underline{v_{n}},\
\overline{u_{n}}\cdot\overline{v_{n}}\big\}\\
&=\|u_{n}v_{n}\|\\
&=\Bigg\|\sum^{n}_{i=1}\nabla u_{i}\Bigg\|
\cdot\Bigg\|\sum^{n}_{i=1}\nabla v_{i}\Bigg\|\\
&\leq \frac{n}{2}\sum^{n}_{i=1}\Big(\big\|\nabla u_{i}\big\|^{2}
+ \big\|\nabla v_{i}\big\|^{2}\Big).
\end{split}
\end{equation*}
This concludes the proof.
\end{proof}

The following results are proved
similarly to Theorem~\ref{thm 3.6}.

\begin{theorem}
\label{thm 3.7}
Let $\{u_{i}\}_{i=0}^{m}$ and $\{v_{i}\}_{i=0}^{m}$ be two sequences of intervals,
$u_{m}=v_{m}=[0,0]$. If $u_{i}$ and $v_{i}$ are synchronous monotone
and $\mu$-decreasing, then
\begin{equation*}
\sum^{m}_{i=n+1}\big\|u_{i-1}\nabla v_{i} + v_{i}\nabla u_{i}\big\|
\leq
\frac{m-n}{2}\sum^{m}_{i=n+1}\big\|(\nabla u_{i})^{2} + (\nabla v_{i})^{2}\big\|.
\end{equation*}
\end{theorem}

\begin{theorem}
\label{thm 3.8}
Let $\{u_{i}\}_{i=0}^{m}$ and $\{v_{i}\}_{i=0}^{m}$ be two sequences of intervals,
$u_{0}=v_{0}=[0,0]$. If $\{u_{i}\}_{i=0}^{m}$ is piecewise alternate monotone,
piecewise alternate $\mu$-monotone, and there is no other point $i$
such that $u_{i}=[0,0]$ and $v_{i}=[0,0]$, then
\begin{equation*}
\sum^{n}_{i=1}\big\|u_{i-1}\nabla v_{i} + v_{i}\nabla u_{i}\big\|
\leq \frac{n}{2}\sum^{n}_{i=1}\big\|(\nabla u_{i})^{2} + (\nabla v_{i})^{2}\big\|.
\end{equation*}
\end{theorem}

\begin{theorem}
\label{thm 3.9}
Let $\{u_{i}\}_{i=0}^{m}$ and $\{v_{i}\}_{i=0}^{m}$ be two sequences of intervals,
$u_{m}=v_{m}=[0,0]$. If $\{u_{i}\}_{i=0}^{m}$ is piecewise alternate monotone,
piecewise alternate $\mu$-monotone, and there is no other point $i$
such that $u_{i}=[0,0]$ and $v_{i}=[0,0]$, then
\begin{equation*}
\sum^{m}_{i=n+1}\big\|u_{i-1}\nabla v_{i} + v_{i}\nabla u_{i}\big\|
\leq
\frac{m-n}{2}\sum^{m}_{i=n+1}\big\|(\nabla u_{i})^{2} + (\nabla v_{i})^{2}\big\|.
\end{equation*}
\end{theorem}

\begin{theorem}
\label{thm 3.10}
Let $\{u_{i}\}_{i=0}^{m}$ and $\{v_{i}\}_{i=0}^{m}$ be two sequences of intervals,
$u_{1}=v_{1}=[0,0]$, and $u_{m}=v_{m}=[0,0]$.
If $\{u_{i}\}_{i=0}^{m}$ is piecewise alternate monotone,
piecewise alternate $\mu$-monotone, and there is no other point $i$
such that $u_{i}=[0,0]$ and $v_{i}=[0,0]$, then
\begin{equation*}
\sum^{m}_{i=1}\big\|u_{i-1}\nabla v_{i} + v_{i}\nabla u_{i}\big\|
\leq \frac{[\frac{m+1}{2}]}{2}\sum^{m}_{i=1}\big\|(\nabla u_{i})^{2}
+ (\nabla v_{i})^{2}\big\|.
\end{equation*}
\end{theorem}

% -----------------------------------------------------------------

\section{Opial type inequalities involving the forward/delta gH-difference operator}
\label{sec:4}

In Section~\ref{sec:3}, we obtained several Opial type inequalities involving
the backward gH-difference operator. Similar arguments can be used to establish
discrete Opial type inequalities concerning the forward gH-difference operator.
The proofs of the results formulated here are left to the interested reader.

\begin{theorem}[delta version of Theorem~\ref{thm 3.1}]
\label{thm 4.1}
Let $\{u_{i}\}_{i=1}^{n}$ be a sequence of intervals, $u_{0}=[0,0]$,
and $\lambda_{1}, \lambda_{2}\geq 1$. If $u_{i}$ is monotone and $\mu$-increasing, then
\begin{equation*}
\sum^{n-1}_{i=0}\Big\|u_{i}^{\lambda_{1}}\big(\Delta u_{i}\big)^{\lambda_{2}}\Big\|
\leq
\frac{\lambda_{2}(n+1)^{\lambda_{1}}}{\lambda_{1}+\lambda_{2}}
\sum^{n-1}_{i=0}\big\|\Delta u_{i}\big\|^{\lambda_{1}+\lambda_{2}}.
\end{equation*}
\end{theorem}

\begin{theorem}[delta version of Theorem~\ref{thm 3.2}]
\label{thm 4.2}
Let $\{u_{i}\}_{i=0}^{m}$ be a sequence of intervals, $u_{m}=[0,0]$,
and $\lambda_{1}, \lambda_{2}\geq1$.
If $u_{i}$ is monotone and $\mu$-decreasing, then
\begin{equation*}
\sum^{m-1}_{i=n}\Big\|u_{i}^{\lambda_{1}}\big(\Delta u_{i}\big)^{\lambda_{2}}\Big\|
\leq \frac{\lambda_{2}(m-n+1)^{\lambda_{1}}}{\lambda_{1}+\lambda_{2}}
\sum^{m}_{i=n}\big\|\Delta u_{i}\big\|^{\lambda_{1}+\lambda_{2}}.
\end{equation*}
\end{theorem}

\begin{theorem}[delta version of Theorem~\ref{thm 3.5}]
\label{thm 4.5}
Let $\{u_{i}\}_{i=0}^{m}$ be a sequence of intervals, $u_{0}=u_{m}=[0,0]$,
and $\lambda_{1},\ \lambda_{2}\geq1$. If $\{u_{i}\}_{i=0}^{m}$ is piecewise alternate monotone,
piecewise alternate $\mu$-monotone, and there is no other point $i$ such that $u_{i}=[0,0]$, then
\begin{equation*}
\sum^{m-1}_{i=1}\Big\|u_{i}^{\lambda_{1}}\big(\Delta u_{i}\big)^{\lambda_{2}}\Big\|
\leq \frac{\lambda_{2}\Big(\big[\frac{m}{2}\big]+1\Big)^{\lambda_{1}}}{\lambda_{1}+\lambda_{2}}
\sum^{m}_{i=1}\big\|\Delta u_{i}\big\|^{\lambda_{1}+\lambda_{2}}.
\end{equation*}
\end{theorem}

% -----------------------------------------------------------------

\section{Conclusions}
\label{sec:5}

We investigated discrete Opial type inequalities for interval-valued functions,
and obtained several new interval discrete Opial type inequalities.
Our results generalize many known discrete Opial type inequalities,
and will be useful in developing the theory of interval difference
inequalities and interval difference equations.
As future research directions, we intend to investigate
interval discrete Opial type inequalities on time scales,
and give some applications to interval difference equations.

% -----------------------------------------------------------------

\section*{Acknowledgements}

This research is supported by Philosophy and Social Sciences 
of Educational Commission of Hubei Province of China (22Y109), 
Foundation of Hubei Normal University (2022055), and Key Projects 
of Educational Commission of Hubei Province of China (D20192501).
Torres is supported by the Portuguese Foundation for Science 
and Technology (FCT), project UIDB/04106/2020 (CIDMA).

% -----------------------------------------------------------------

% -----------------------------------------------------------------

\end{document}